\documentclass{amsart}
\usepackage{amssymb}

\newtheorem{thm}{Theorem}[section]

\newtheorem{prp}[thm]{Proposition}

\theoremstyle{definition}
\newtheorem{dfn}[thm]{Definition}
\newtheorem{rem}[thm]{Remark}
\newtheorem{exa}[thm]{Example}

\def\LTs{\textnormal{LT}_\sigma}
\def\MT{\textnormal{MT}}
\def\Mat{\textnormal{Mat}}
\def\Supp{\textnormal{Supp}}
\def\DegLex{\texttt{DegLex} }
\def\Lex{\texttt{Lex} }
\def\MatA{\mathcal{A}}
\def\FF{\mathcal{F}}
\def\GG{\mathcal{G}}
\def\OO{\mathcal{O}}
\def\OOsI{\mathcal{O}_\sigma\{I\}}
\def\UU{\mathcal{U}}
\def\VV{\mathcal{V}}
\def\WW{\mathcal{W}}
\def\NN{\mathbb{N}}
\def\QQ{\mathbb{Q}}
\def\TT{\mathbb{T}}

\def\cocoa{\mbox{\rm
 C\kern-.13em o\kern-.07 em C\kern-.13em o\kern-.15em A}}
\def\apcocoa{\mbox{\rm
A\kern-0.13em p\kern -0.07em C\kern-.13em o\kern-.07 em C\kern-.13em
o\kern-.15em A}}

\begin{document}

\title{Computing Border Bases without using a Term Ordering}

\author{Stefan Kaspar}
\address{Fakult\"at f\"ur Informatik und Mathematik, Universit\"at
Passau, D-94030 Passau, Germany}
\email{Stefan.Kaspar@uni-passau.de}

\date{\today}

%%%%%%%%%%%%
% Abstract %
%%%%%%%%%%%%
\begin{abstract}
Border bases, a generalization of Gr\"{o}bner bases, have actively been
researched during recent years due to their applicability to industrial
problems. In \cite{KKCompBBases} Kehrein and Kreuzer formulated the so called
\textit{Border Basis Algorithm}, an algorithm which allows the computation of
border bases that relate to a degree compatible term ordering. In this paper
we extend the original Border Basis Algorithm in such a way that also border
bases that do not relate to any term ordering can be computed by it.
\end{abstract}


\maketitle

%%%%%%%%%%%%%%%%
% Introduction %
%%%%%%%%%%%%%%%%
\section{Introduction}\label{SecIntroduction}
Let $P = K[x_1, \dots, x_n]$ be a polynomial ring over a field $K$ and $I
\subset P$ be a zero-dimensional ideal. Let $\OO = \{ t_1, \dots, t_m \}$ be a
set of terms that is closed under taking divisors. Such a set is called order
ideal and its border $\{ b_1, \dots, b_s \} = (x_1\OO \cup \ldots \cup x_n\OO)
\setminus \OO$ is denoted by $\partial \OO$. A set $\GG = \{ g_1, \dots, g_s \}
\subset P$ of polynomials is then called an $\OO$-border basis of $I$ if $g_j =
b_j - \sum _{i=1}^m c_{ji} t_i$ where $c_{ji} \in K, \GG$ generates $I$, and
$\OO$ is a $K$-vector space basis of $P/I$. Border bases have actively been
researched during recent years (see for instance
\cite{HKPPAppCompZDPolyIdeals}, \cite{KKCharOfBBases}, \cite{KKCompBBases},
\cite{KKRAlgViewBBases}, and \cite{MNewCritForNormFormAlgo}) which is due
to several reasons. One aspect is that border bases are a generalization of
Gr\"{o}bner bases. If we denote by $\OOsI$ the complement of the set of leading
terms of $I$ with respect to any term ordering $\sigma$ then there exists an
(unique) $\OOsI$-border basis of $I$ and this border basis also contains the
reduced $\sigma$-Gr\"{o}bner basis of $I$. In addition border bases not only
generalize Gr\"{o}bner bases but are also more suitable for computations that
deal with empirical polynomials, i.e. polynomials that were constructed from
measured data (cf. \cite{SNumPolAlg}). This advantage over Gr\"{o}bner basis
has even lead to industrial applications (cf. \cite{HKPPAppCompZDPolyIdeals}).

\smallskip
The natural question is how border bases can actually be computed. In
\cite{KKCompBBases} Kehrein and Kreuzer show how it is possible to compute any
$\OO$-border basis of an ideal if the order ideal $\OO$ is already known. The
downside of this approach is that it involves the computation of a Gr\"{o}bner
basis of the input ideal and that $\OO$ has to be be specified explicitly. In
\cite{MNewCritForNormFormAlgo} Mourrain proposed a generic framework for
computing quotient bases and generators of an ideal. This framework has been
refined by Kehrein and Kreuzer in \cite{KKCompBBases}. There, the so called
\textit{Border Basis Algorithm} is formulated which allows the explicit
computation of $\OOsI$-border bases where $\sigma$ is a degree compatible term
ordering (cf. \cite[Proposition 18]{KKCompBBases}). This algorithm is the basis
for the work in this paper. We deduce an extension of the Border Basis
Algorithm which allows us to compute more general border bases that do not rely
on a term ordering. To do so we shortly revisit the necessary theoretical
background about border bases at the beginning of section
\ref{SecBasicDefsAndRes}. Then we collect some results from \cite{KKCompBBases}
which we need to formulate our Border Basis Algorithm generalization and give
an example of a border basis that cannot be computed by the Border Basis
Algorithm of Kehrein and Kreuzer. In section \ref{SecBBAWithTMS} we first
formalize the concept of marking polynomials. This concept allows us to speak
about the marked term of a polynomial and replaces the usage of a degree
compatible term ordering in the original Border Basis Algorithm. We then
reformulate those parts of the original Border Basis Algorithm that need to be
adapted to the concept of marked polynomials. Finally we propose a new Border
Basis Algorithm that does not use a degree compatible term ordering for the
computation of a border basis. We show that it really is an extension of the
original Border Basis Algorithm and analyze its properties. In the final
section \ref{SecRelToOtherAlgos} we put our new Border Basis Algorithm into
context with the approaches to compute border bases of Mourrain and
Tr\'{e}buchet in \cite{MTGenNormFormAndPolySysSolv} and Braun and Pokutta in
\cite{BPBBasesAndOIPolyhedralChar}.

\smallskip
\textbf{Acknowledgement.} The author would like to thank Prof. Dr. Martin
Kreuzer of the Department of Informatics and Mathematics, Universit\"{a}t
Passau, for fruitful discussions and helpful suggestions with respect to this
paper.

\bigskip

%%%%%%%%%%%%%%%%%%%%%%%%%%%%%%%%%
% Basic definitions and results %
%%%%%%%%%%%%%%%%%%%%%%%%%%%%%%%%%
\section{Basic Definitions and Results}\label{SecBasicDefsAndRes}
In the following let $K$ be a field, $n \geq 1$, $P =
K[x_1, \dots, x_n]$, and $\TT^n = \{ x_1^{\alpha_1} \cdots x_n^{\alpha_n}\ |\
\alpha_i \geq 0, 1 \leq i \leq n \}$. Let us recall the definition of a border
basis and some basic results of border basis theory as stated in
\cite[Section 6.4]{KRBook2}.

\begin{dfn}
A set $\OO \subseteq \TT^n$ is called \textbf{order ideal} if $t' \in \OO$
and $t|t'$ imply $t \in \OO$ where $t, t' \in \TT^n$. The set $\partial \OO
= (x_1\OO \cup \ldots \cup x_n\OO) \setminus \OO$ is called the
\textbf{border of $\OO$}. We let $\partial \emptyset = \{ 1 \}$.
\end{dfn}

Please note that in the following we always consider finite order ideals $\OO$
only.

\begin{dfn}\label{DfnOBB}
Let $\OO = \{ t_1, \dots, t_\mu \} \subset \TT^n$ be an order ideal,
$\partial \OO = \{ b_1, \dots, b_\nu \}$,\linebreak
$G = \{ g_1, \dots, g_m \} \subset P$, and $I \subseteq P$ be
an ideal.
\begin{enumerate}
  \item[a)] The set $G$ is called \textbf{$\OO$-border prebasis} if
  \[ g_j = b_j - \sum_{i = 1}^\mu \alpha_{ij} t_i \]
  where $\alpha_{ij} \in K$ for $1 \leq i \leq \mu$ and
  $1 \leq j \leq \nu$.
  \item[b)] An $\OO$-border prebasis $G$ is called
  \textbf{$\OO$-border basis of $I$} if
  $G \subset I$ and one of the following equivalent conditions is
  satisfied:
  \begin{enumerate}
    \item[(i)] $\overline{\OO} = \{ \overline{t_1}, \dots,
    \overline{t_\mu} \}$ is a $K$-vector space basis of $P/I$.
    \item[(ii)] $I \cap \langle \OO \rangle_K = \{ 0 \}$.
    \item[(iii)] $P = I \oplus \langle \OO \rangle_K$.
  \end{enumerate}
  In this case $I$ is necessarily a zero-dimensional ideal.
\end{enumerate}
\end{dfn}

\begin{prp}\label{PrpExAndUniqOfBB}
Let $\OO = \{ t_1, \dots, t_\mu \} \subset \TT^n$ be an order ideal,
$I \subseteq P$ be a zero-dimensional ideal and assume that
$\overline{\OO}$ is a $K$-vector space basis of $P/I$.
\begin{enumerate}
  \item[a)] There exists a unique $\OO$-border basis of $I$.
  \item[b)] If $G$ is an $\OO$-border prebasis and $G \subset I$
  then $G$ is the $\OO$-border basis of $I$.
  \item[c)] In the setting of b), $I$ is generated by $G$.
\end{enumerate}
\end{prp}
% \begin{proof}
% See [KR2], Proposition 6.4.15 and Proposition 6.4.17: Existence in
% a) by existence of $\OO$-border prebasis and definition; uniqueness
% by decomposition $P = I \oplus \langle \OO \rangle_K$ and contradiction.
% b) by a). c) by Border Division Algorithm 6.4.11.
% \end{proof}

% \begin{prp}\label{PrpBBCritForBB}
% \textnormal{\textbf{(Buchberger Criterion for Border Bases)}}\\
% Let $\OO = \{ t_1, \dots, t_\mu \} \subset \TT^n$ be an order ideal,
% $\partial \OO = \{ b_1, \dots, b_\nu \},$\linebreak
% $G = \{ g_1, \dots, g_m \}
% \subset P$ be an $\OO$-border prebasis, and $I \subseteq P$ be an ideal.
% $g_k, g_l$ are \textbf{neighbors} if $x_i b_k = b_l$ or $x_i b_k = x_j b_l$
% for some indeterminates $x_i, x_j$. The corresponding \textbf{S-polynomials}
% are
% \[ S(g_k, g_l) = x_i g_k - g_l \quad \textnormal{and} \quad S(g_k, g_l) = x_i g_k - x_j g_l \]
% respectively.\\
% In this setting, $G$ is the $\OO$-border basis of $I$ if and only if
% $G \subset I$ and for each pair of neighboring $g_k, g_l$ there
% are (constant!) $c_j \in K$ such that
% \[ S(g_k, g_l) = c_1 g_1 + \dots + c_\nu g_\nu. \]
% \end{prp}
% \begin{proof}
% See [KR2], Proposition 6.4.34: ''$\Rightarrow$'' by 6.4.28 (rewrite
% relations), ''$\Leftarrow$'' by 6.4.30 (commuting matrices).
% \end{proof}
The following definition and proposition are taken from \cite{KKCompBBases}.
They will be used to state and derive the results given in in section
\ref{SecBBAWithTMS}.

\begin{dfn}
Let $F \subset P$ and $V$ be a $K$-vector subspace of $P$.
\begin{enumerate}
  \item[a)] $F^+ = F \cup x_1 F \cup \dots \cup x_n F$.
  \item[b)] $V^+ = V + x_1 V + \dots + x_n V$.
\end{enumerate}
\end{dfn}

\begin{prp}\label{PrpExistenceOfBB}
Let $\UU$ be a $K$-vector subspace of $P$. Let $\VV$ be a $K$-vector
subspace of a zero-dimensional ideal $I \subseteq P$ such that $\VV^+
\cap\ \UU = \VV$ and $\langle \VV \rangle = I$. Let\linebreak
$\OO = \{ t_1, \dots, t_\mu \} \subset \TT^n$ be an order ideal which
satisfies
\[  \UU = \VV \oplus \langle \OO \rangle_K. \]
If $\partial \OO \subseteq \UU$ then there exists an $\OO$-border basis of
$I$.
\end{prp}

The following example shows in which way the standard Border Basis Algorithm as
given in \cite{KKCompBBases} is limited and points out the cause of the
limitation. The following sections will then address this problem.

\begin{exa}\label{ExaStdBBAProblem}
Let $K = \QQ$ and $P = \QQ[x, y]$. Let $\OO = \{ 1, y, y^2, x, x^2 \}$.
The set $\{ g_1, g_2, g_3, g_4, g_5 \} = \{ xy + x^2 - \frac{1}{2}y^2 - x
- \frac{1}{2}y, x^3 - x, x^2y - \frac{1}{2}y^2 - \frac{1}{2}y, xy^2 + x^2
- \frac{1}{2}y^2 - x - \frac{1}{2}y, y^3 - y \}$
is the $\OO$-border basis of the vanishing ideal $I$ of the
five points $(-1, 1), (1, 1)$,
$(0, 0), (1, 0), (0, -1) \in \mathbb{A}^2(\QQ)$.
But it cannot be computed by the (Improved) Border Basis
Algorithm as stated in \cite{KKCompBBases} because
 $\LTs(g_1) = x^2$ or $\LTs(g_1) = y^2$ for any term ordering $\sigma$ on
$\TT^2$ which implies $\LTs(I)
\cap \OO \ne \emptyset$.
\end{exa}

\bigskip

%%%%%%%%%%%%%%%%%%%%%%%%%%%%%%%%%%%%%%%%%%%%%%%%%%%%%%%
% A Border Basis Algorithm with Term Marking Strategy %
%%%%%%%%%%%%%%%%%%%%%%%%%%%%%%%%%%%%%%%%%%%%%%%%%%%%%%%
\section{A Border Basis Algorithm with Term Marking
Strategy}\label{SecBBAWithTMS}
As we have seen in example \ref{ExaStdBBAProblem}, using a (degree
compatible) term ordering on $\TT^n$ may prevent certain terms in
the support of a polynomial from becoming a border term. In the
following we investigate the idea of exchanging the term ordering for a
weaker marking on the polynomials while keeping the restriction that a term $t
\in \Supp(f)$ of a polynomial $f$ may only be marked if $\deg(t) = \deg(f)$.
This idea will then be incorporated into the Improved Border Basis Algorithm
\cite[Proposition 21]{KKCompBBases}. As a consequence we will have to
reformulate those steps of the Improved Border Basis Algorithm that rely on the
usage of a (degree compatible) term ordering. Let us first specify which kind
of polynomial marking we want to consider.

% Definition of marked polynomials
\begin{dfn}
Let $\FF \subseteq P \setminus \{ 0 \}$.
\begin{enumerate}
  \item[a)] We say that the polynomials of $\FF$ are \textbf{marked} if there
  exists a mapping
  $\MT_\FF : \FF \rightarrow \bigcup _{f \in \FF} \Supp(f)$
  which satisfies $\MT_\FF(f) \in \Supp(f)$ and
  $\deg(\MT_\FF(f)) = \deg(f)$ for all $f \in \FF$.
  We call $\MT_\FF(f)$ the \textbf{marked term} of $f \in
  \FF$ and $\MT_\FF$ a \textbf{marking} of $\FF$.
  \item[b)] If $\MT_\FF$ is a marking of $\FF$ then we say that $t \in
  \bigcup _{f \in \FF} \Supp(f)$ is \textbf{marked} if $t = \MT_\FF(f)$
  for an $f \in \FF$.
\end{enumerate}
\end{dfn}

\begin{exa}\label{ExaDegCompTOIsMarking} Let $\sigma$ be a degree compatible
term ordering on $\TT^n$ and let $\LTs(f)$ denote the leading term of a
polynomial $f \in P$ with respect to $\sigma$. Then $\LTs$ is a marking for any
$\FF \subseteq P \setminus \{ 0 \}$.
\end{exa}

\begin{exa} Let $K = \QQ, P = \QQ[x, y]$, and $\FF = \{ f_1, f_2 \} = \{ x^2 +
xy - y^2 - x - y, x^3 - x \}$. If we want to mark the polynomials of $\FF$ we
have to mark the term $x^3$ in $f_2$. For the marked term of $f_1$ we have
three different possibilities to choose from: we could either mark $x^2, y^2$
or $xy$. Please note that the term $xy$ could not appear as a marked term if we
would choose a degree compatible term ordering on $\TT^2$ to mark the
polynomials of $\FF$ as explained in example \ref{ExaDegCompTOIsMarking}.
\end{exa}

% Marked Interreduction
Our next goal is to adapt the computation procedure \cite[Lemma
12]{KKCompBBases} to our setting of marked polynomials.

\begin{prp}\label{PrpMarkedInterreduction}\textnormal{\textbf{(Marked
Interreduction)}}\\
Let $\{ f_1, \ldots, f_s \} \subset P \setminus \{ 0 \}$ be a
set of marked polynomials. Let $\mathop{\bigcup}_{i=1}^s
\Supp(f_i) = \{ t_1, \ldots, t_l \}$ be enumerated in such a way that
$\deg(t_1) \geq \ldots \geq \deg(t_l)$ holds. Consider the following
algorithm.
\begin{enumerate}
  \item[(1)] Write $f_i = a_{i1} t_1 + \ldots + a_{il} t_l$ for $i = 1, \ldots,
  s$ where $a_{ij} \in K$ and let $\MatA := (a_{ij}) \in \Mat_{s,l}(K)$.
  \item[(2)] Let $p_1, \ldots, p_s \in \{ 1, \ldots, l \}$ such that
  $\MT_{\FF}(f_i) = t_{p_{_i}}$ for $i = 1, \ldots, s$.
  \item[(3)] Repeat the following steps $(4)-(6)$ for $i = 1, \ldots, s$.
  \item[(4)] If $a_{ip_{_i}} \ne 0$ then replace $a_{ij}$ by
  $a_{ip_{_i}}^{-1} \cdot a_{ij}$ for $j = 1, \ldots, l$.
  \item[(5)] For each $k \in \{ 1, \ldots, s \} \setminus \{ i \}$ such that
  $a_{kp_{_i}} \ne 0$ replace $a_{kj}$ by $a_{kj} - a_{kp_{_i}}
  \cdot a_{ij}$ for $j = 1, \ldots, l$.
  \item[(6)] Check for $k = i + 1, \ldots, s$ if $a_{kp_{_k}} = 0$. For any such
  $k$ determine the minimum column index $j$ such that $a_{kj} \ne 0$ if such an
  index exists and set $p_k := j$.
  \item[(7)] Let $\FF' = \emptyset$. For $i = 1, \ldots, s$ check if $f_i :=
  a_{i1}t_1 + \ldots + a_{il}t_l \ne 0$. For any such $f_i$ set $\FF' := \FF'
  \cup \{ f_i \}$ and $\MT_{\FF'}(f_i) := t_{p_{_i}}$.
  \item[(8)] Return $\FF'$ and $\MT_{\FF'}$.
\end{enumerate}
This is an algorithm that computes a set of polynomials $\FF' = \{ f'_1, \ldots,
f'_m \}$ and a mapping $\MT_{\FF'} : \FF' \rightarrow \bigcup _{i = 1}^m
\Supp(f'_i)$ such that the following properties hold.
\begin{enumerate}
  \item[a)] $\langle f'_1, \ldots, f'_m \rangle_K = \langle f_1, \ldots f_s
  \rangle_K$.
  \item[b)] $\MT_{\FF'}$ is a marking of $\FF'$.
  \item[c)] $\MT_{\FF'}(f'_i) \notin \Supp(f'_j)$ for all $i, j \in \{ 1,
  \ldots, m \}$ and $i \ne j$. In particular, the marked terms of $f'_1,
  \ldots, f'_m$ are pairwise different.
\end{enumerate}
\end{prp}
\begin{proof}
All loops contained in the algorithm terminate obviously after finitely many
iterations.

Claim a) holds because the only changes performed on a input polynomial
are the addition of $K$-multiples of some other input polynomials or the scaling
of the polynomial by a factor in $K$.\\
To prove claim b) we show that after each iteration of the main loop $(4)-(6)$
the equation $\deg(t_{p_{_j}}) = \deg(a_{j1} t_1 + \ldots + a_{jl} t_l)$ and
the inequation $a_{jp_{_j}} \ne 0$ hold for all $j \in \{ 1, \ldots, s
\}$ such that $a_{j1} t_1 + \ldots + a_{jl} t_l \ne 0$. First we note that upon
termination of the steps $(4)-(6)$ the $p_i$-th column of $\MatA$ contains
exactly one non-zero element $a_{ip_{_i}}$ if the $i$-th row of $\MatA$ is not
a zero row. By initialization $\deg(a_{11} t_1 + \ldots + a_{1l} t_l) =
\deg(t_{p_{_1}})$ and $a_{1p_{_1}} \ne 0$ hold during the first iteration of
loop $(4)-(6)$. Assume a row $k$ is going to be reduced and it does not
completely reduce to zero. Before the reduction the relation $\deg(a_{k1} t_1 +
\ldots + a_{kl} t_l) \geq \deg(t_{p_{1}})$ holds. If $a_{kp_{_k}} \ne 0$ after
the reduction then the equality $\deg(a_{k1} t_1 + \ldots + a_{kl} t_l) =
\deg(t_{p_{_k}})$ follows directly. Otherwise the equality and $a_{kp_{_k}} \ne
0$ are implied by the decreasing ordering of the terms $t_1, \ldots, t_l$ and
the update of $p_k$. Now assume the $i$-th iteration of the loop $(4)-(6)$ is
being executed for $i \geq 2$. A reduction of a row $k$ in step $(5)$ where $k <
i$ does not affect the element $a_{kp_{_k}}$ because of $a_{ip_{_k}} = 0$. Thus
the statements $\deg(a_{k1} t_1 + \ldots + a_{kl} t_l) = \deg(t_{p_{_k}})$ and
$a_{kp_{_k}} \ne 0$ are true. With an argument similar to the one in the case of
the first iteration of the loop $(4)-(6)$ we derive the conclusion that the
equation $\deg(a_{k1} t_1 + \ldots + a_{kl} t_l) = \deg(t_{p_{_k}})$ and the
inequation $a_{kp_{_k}} \ne 0$ hold for $k > i$ upon termination of step $(6)$
whenever $a_{k1} t_1 + \ldots + a_{kl} t_l \ne 0$. This yields
the correctness of claim b).

To prove the correctness of claim c) consider $i \in \{ 1, \ldots, s \}$ such
that the $i$-th row of $\MatA$ is not a zero row. From our observations above
we conclude that $a_{ip_{_i}}$ is the only non-zero element of the
$p_i$-th column of $\MatA$. As $a_{ip_{_i}}$ is the coefficient of the
marked term of $a_{i1} t_1 + \ldots + a_{il} t_l$ with respect to $\MT_{\FF'}$,
claim c) is true.
\end{proof}

\begin{rem}
Obviously the resulting marking $\MT_{\FF'}$ does not only depend on the
initial marking $\MT_{\FF}$ but also on the ordering of the terms $t_1, \ldots,
t_l$. To gain more control over the resulting marking $\MT_{\FF'}$, step $(5)$
can be replaced by a refined pivot element selection process.
\end{rem}

\begin{rem}\label{RemMIBacktracking}
If the value of a pivot index $p_k$ is updated during the execution of the
algorithm in step (6) the situation might occur that the set\linebreak
$\{ i\ |\
j \leq i \leq l\ \textnormal{and}\ \deg(t_i) = \deg(t_j) \}$ contains more than one
element. This means that $p_k$ could be set to any value of this set since it
would not change the correctness of the algorithm. If the algorithm is equipped
with an additional book keeping functionality that creates a log entry whenever
such a situation occurs, it is possible to perform backtracking to redo the
computation with different pivot indices.
\end{rem}

\begin{exa}
Let $P = \QQ[x, y, z]$ and $\FF = \{ f_1, f_2, f_3, f_4, f_5 \} = \{ xy^2 +
x^3 + z, xy^2 - yz, xy^2 + 1, y^5 - xy^2 - y^2 + yz + z, z^3 \}$ a set of
marked polynomials where $\MT_\FF(f_1) = \MT_\FF(f_2) = \MT_\FF(f_3) = xy^2,
\MT_\FF(f_4) = y^5$, and $\MT_\FF(f_5) = z^3$. We choose the numeration
$t_1 = y^5, t_2 = z^3, t_3 = x^3, t_4 = xy^2, t_5 = yz, t_6 = y^2, t_7 = z$,
and $t_8 = 1$ and apply the Marked Interreduction.
\begin{enumerate}
  \item[(1)] We obtain the matrix
  \[ \MatA =
    \begin{pmatrix}
      0 & 0 & 1 & [1] & 0 & 0 & 1 & 0\\
      0 & 0 & 0 & [1] & -1 & 0 & 0 & 0\\
      0 & 0 & 0 & [1] & 0 & 0 & 0 & 1\\
      [1] & 0 & 0 & -1 & 1 & -1 & 1 & 0\\
      0 & [1] & 0 & 0 & 0 & 0 & 0 & 0
    \end{pmatrix}
  \]
  where the elements in brackets mark the elements $a_{ip_{_i}}$ in each row.
  \item[(2)] We obtain $p_1 = p_2 = p_3 = 4, p_4 = 1$, and $p_5 = 2$.
  \item[(3)] Since $a_{1p_{_1}} = 1$ we do not change any element of the
  first row of $\MatA$.
  \item[(4)] Since $a_{2p_{_1}}, a_{3p_{_1}}$, and $a_{4p_{_1}}$ are not equal
  to zero the matrix $\MatA$ becomes the matrix
  \[ \MatA =
    \begin{pmatrix}
      0 & 0 & 1 & [1] & 0 & 0 & 1 & 0\\
      0 & 0 & -1 & [0] & -1 & 0 & -1 & 0\\
      0 & 0 & -1 & [0] & 0 & 0 & -1 & 1\\
      [1] & 0 & 1 & 0 & 1 & -1 & 2 & 0\\
      0 & [1] & 0 & 0 & 0 & 0 & 0 & 0
    \end{pmatrix}.
  \]
  \item[(5)] $a_{2p_{_2}} = a_{3p_{_3}} = 0$ and we set $p_2 = p_3 = 3$.
  \item[(3)] We multiply the second row of $\MatA$ by $-1$.
  \item[(4)] Since $a_{3p_{_2}} \ne 0$ the matrix $\MatA$ becomes the matrix
  \[ \MatA =
    \begin{pmatrix}
      0 & 0 & 1 & [1] & 0 & 0 & 1 & 0\\
      0 & 0 & [1] & 0 & 1 & 0 & 1 & 0\\
      0 & 0 & [0] & 0 & 1 & 0 & 0 & 1\\
      [1] & 0 & 1 & 0 & 1 & -1 & 2 & 0\\
      0 & [1] & 0 & 0 & 0 & 0 & 0 & 0
    \end{pmatrix}.
  \]
  \item[(5)] $a_{3p_{_3}} = 0$ and we set $p_3 = 5$.
  \item[(3)] Since $a_{3p_{_3}} = 1$ we do not change any element of the
  third row of $\MatA$.
  \item[(4)] Since $a_{2p_{_3}}$ and $a_{4p_{_3}}$ are not equal
  to zero the matrix $\MatA$ becomes the matrix
  \[ \MatA =
    \begin{pmatrix}
      0 & 0 & 1 & [1] & 0 & 0 & 1 & 0\\
      0 & 0 & [1] & 0 & 0 & 0 & 1 & -1\\
      0 & 0 & 0 & 0 & [1] & 0 & 0 & 1\\
      [1] & 0 & 1 & 0 & 0 & -1 & 2 & -1\\
      0 & [1] & 0 & 0 & 0 & 0 & 0 & 0
    \end{pmatrix}.
  \]
\end{enumerate}
The matrix $\MatA$ is left unchanged during the remaining loops of the
steps $(3)-(5)$ where $i = 4$ and $i = 5$. Finally, in step $(7)$ we obtain
$\FF' = \{ f'_1, f'_2, f'_3, f'_4, f'_5 \} = \{ xy^2 + x^3 + z, x^3 + z - 1, yz
+ 1, y^5 + x^3 - y^2 + 2z - 1, z^3 \}$ where $\MT_{\FF'}(f'_1) = xy^2,
\MT_{\FF'}(f'_2) = x^3, \MT_{\FF'}(f'_3) = yz, \MT_{\FF'}(f'_4) = y^5,
\MT_{\FF'}(f'_5) = z^3$.
\end{exa}

% Border Basis Algorithm with Term Marking Strategy
Let $\FF = \{ f_1, \ldots, f_s \} \subset P \setminus \{ 0 \}$ and assume that
$I = \langle \FF \rangle_P$ is zero-dimensional. We are now ready to
reformulate the Improved Border Basis Algorithm \cite[Proposition
21]{KKCompBBases}. The result will be an algorithm that allows the user to
outline a subset of $\TT^n$ by choosing a marking on $\FF$. The algorithm
will then try to compute a suitable order ideal $\OO$ in this outlined subset
of $\TT^n$ for which an $\OO$-border basis of $I$ exists. This flexibility is
achieved at the cost of a loss of predictability: The algorithm will not be
able to produce an $\OO$-border basis of $I$ for any given marking on $\FF$ (cf.
remark \ref{RemAlgoCanStopInT7}). However, if a degree compatible term ordering
$\sigma$ on $\TT^n$ is used to mark the polynomials of $\FF$ and the polynomials
obtained during the computation then the algorithm will always compute the
$\OOsI$-border basis of $I$ (cf. remark \ref{RemDCTOAlwaysWorks}). Lastly, we
note that the overall structure of the Improved Border Basis Algorithm
\cite[Proposition 21]{KKCompBBases} will be kept unchanged in the following
reformulation.

\begin{prp}\label{PrpBBAWithTMS}\textnormal{\textbf{(Border Basis Algorithm
with Term Marking Strategy)}}\\
Let $\FF = \{ f_1, \ldots, f_s \} \subset P \setminus \{ 0 \}$ be a set of
marked polynomials such that $\MT_{\FF}(f_i) \notin \Supp(f_j)$ where $1 \leq i,
j \leq s$ and $i \ne j$. Assume that $I = \langle \FF \rangle_P$ is
zero-dimensional. The following algorithm stops without a result in step (T7)
or computes an order ideal $\OO \subset \TT^n$ and a set of marked polynomials
$\{ g_1, \ldots, g_\nu \}$ such that $\{ g_1, \ldots, g_\nu \}$ is the
$\OO$-border basis of $I$ and if $t \in \Supp(g)$ satisfies $t \in \partial
\OO$ for a $g \in \{ g_1, \ldots, g_\nu \}$ then $t$ is marked in $g$.

\begin{enumerate}
  \item[(T1)] Let $\,\UU$ be the order ideal spanned by
  $\bigcup_{i=1}^s \Supp(f_i)$.
  \item[(T2)] Compute a $K$-vector space basis $\VV$ of $\langle \FF
  \rangle_K$ with pairwise different marked terms:
  Apply the Marked Interreduction to $\{ f_1, \ldots, f_s \}$ to obtain
  $\VV = \{ f'_1, \ldots, f'_m \}$.
  \item[(T3)] Compute a $K$-vector space basis $\VV' \cup \WW'$ of
  $\langle \VV^+ \rangle_K$ such that the elements of $\VV' \cup \WW'$
  have pairwise different marked terms and $\MT_{\VV' \cup \WW'}(v) \notin
  \Supp(w)$ for all $v, w \in \VV' \cup \WW'$ and $v \ne w$:
  \begin{enumerate}
    \item[a)] Mark in all $x_iv \in \VV^+ \setminus \VV, 1 \leq i \leq n,
    v \in \VV$ the term $x_it$ where $t \in \Supp(v)$ is the marked term of $v$.
    \item[b)] Let $\VV = \{ v_1, \ldots, v_l \}$ and $\VV^+ \setminus \VV = \{
    v'_1, \ldots, v'_{l'} \}$ and apply the Marked Interreduction to $\{ v_1,
    \ldots, v_l, v'_1, \ldots, v'_{l'} \}$ to obtain $\tilde{\VV} = \{
    \tilde{f}_1, \ldots, \tilde{f}_m \}$.
    \item[c)] Let $T = \{ \MT_\VV(v)\ |\ v \in \VV \}$ and $\VV' = \{ \tilde{f}
    \in \tilde{\VV}\ |\ \MT_{\tilde{\VV}}(\tilde{f}) \in T \}$. Let $\WW' =
    \tilde{\VV} \setminus \VV'$.
  \end{enumerate}
  \item[(T4)] Let $\WW = \{ w \in \WW'\ |\ \MT_{\WW'}(w) \in \UU \}$.
  \item[(T5)] If $\ \bigcup_{w \in \WW} \Supp(w) \nsubseteq
  \UU$ then replace $\,\UU$ by the order ideal spanned by $\,\UU$ and\linebreak
  $\bigcup_{w \in \WW} \Supp(w)$ and continue with (T4).
  \item[(T6)] If $\WW \ne \emptyset$ then replace $\VV$ by $\VV' \cup \WW$ and
  continue with (T3).
  \item[(T7)] Let $\mathcal{O} = \UU \setminus \{ \MT_\VV(v)\ |\ v \in \VV \}$.
  If $\OO$ is not an order ideal then stop and output
  ''$\OO$ is not an order ideal in step (T7).''.
  \item[(T8)] If $\partial \OO \nsubseteq \UU$ then replace $\,\UU$ by the
  order ideal $\,\UU^+$ and continue with (T3).
  \item[(T9)] Select in $\VV$ those $g_1, \ldots, g_\nu$ which satisfy
  $\MT_\VV(g_i) \in \partial \OO$ where $1 \leq i \leq \nu$. Output $\{ g_1,
  \ldots, g_\nu \}$ and its marking as well as $\OO$.
\end{enumerate}
\end{prp}
\begin{proof}
The correctness of (T2) and (T3) is implied by proposition 
\ref{PrpMarkedInterreduction}. Additionally, the following equality holds
in (T3): $\bigcup_{v \in \VV} \{ \MT_{\VV}(v) \}
= \bigcup_{v \in \VV'} \{ \MT_{\VV'}(v) \}$. The relation ''$\supseteq$'' holds
by construction of $\VV'$ and ''$\subseteq$'' because of the order in which the
input polynomials $\{ v_1, \ldots, v_l, v'_1, \ldots, v'_{l'} \}$ are
processed by the Marked Interreduction.

The loop (T4)-(T5) is finite since each
enlargement of $\UU$ is contained in the finite order ideal spanned by
$\bigcup_{v \in \VV' \cup \WW'} \Supp(v)$ because of $\WW \subseteq \WW'$.

At the end of loop (T4)-(T5), $\langle \VV' \cup \WW \rangle_K =
\langle \VV' \cup \WW' \rangle_K \cap \langle \UU \rangle_K$ holds:
The relation ''$\subseteq$'' follows by construction of $\WW$ and step (T5). To
show ''$\supseteq$'' let $v = \alpha_1 v_1 + \ldots + \alpha_r v_r
+ \beta_1 w_1 + \ldots + \beta_l w_l \in \langle \UU \rangle_K$ where
$\alpha_1, \ldots, \alpha_r, \beta_1, \ldots, \beta_l \in K \setminus \{ 0 \},
v_1, \ldots, v_r \in \VV'$, and $w_1, \ldots, w_l \in \WW'$. If $r + l = 1$,
the inclusion follows by $\MT_{\VV' \cup \WW'}(v) \in \UU$. Now let
$r + l > 1$. If $r \geq 1$ then $v - \alpha_1 v_1 \in \langle \VV'
\cup \WW' \rangle_K \cap \langle \UU \rangle_K$ because of
$\Supp(v_1) \subseteq \UU$
%%\footnote{$v_1 = v' + \gamma_1 w'_1 + \ldots
%%+ \gamma_a w'_a$ where $v' \in \VV, \gamma_1, \ldots, \gamma_a \in K$ and
%%$w'_1, \ldots, w'_a \in \WW'$. Since the marked terms of $v_1, w'_1, \ldots,
%%w'_a$ are pairwise different, all marked terms of $w'_1, \ldots, w'_a$ must be
%%contained in $\Supp(v) \subseteq \UU$ and therefore $\Supp(v_1) \subseteq \UU$
%%because step (T5) has been executed.}
and by the induction hypothesis
$v = (v - \alpha_1 v_1) + \alpha_1 v_1 \in \langle \VV' \cup \WW \rangle_K$.
If $l \geq 1$ then $v - \beta_1 w_1 \in \langle
\VV' \cup \WW' \rangle_K \cap \langle \UU \rangle_K$ holds: The
term $\MT_{\WW'}(w_1)$ neither is contained in $\Supp(v_i), 1 \leq i \leq r$ nor
$\Supp(w_i), 2 \leq i \leq l$. This implies $\MT_{\WW'}(w_1) \in \UU$ and
therefore $w_1 \in \WW$ since (T4) has been executed and $v - \beta_1 w_1
\in \langle \VV' \cup \WW' \rangle_K \cap \langle \UU \rangle_K$ since (T5)
has been executed. Again, by induction hypothesis $v = (v - \beta_1 w_1) + \beta_1 w_1
\in \langle \VV' \cup \WW \rangle_K$.

The loop (T3)-(T6) is finite: At the beginning of an arbitrary iteration,
let $\UU$ be contained some $\TT^n_{\leq d}$. A possible enlargement of $\UU$
in (T3) is contained in $\TT^n_{\leq d}$. The subset selection
criterion $\MT_{\WW'}(w) \in \UU$ in (T4) yields
$\Supp(w) \subseteq \TT^n_{\geq d}$ where $w \in \WW'$ since
$\deg(\MT_{\WW'}(w)) = \deg(w)$ holds. Thus all enlargements of $\UU$
take place in the finite set $\TT^n_{\leq d}$.

Upon termination of loop (T3)-(T6), the equality $\langle \VV \rangle_K
= \langle \VV \rangle^+_K \cap \langle \UU \rangle_K$ holds: (T3) yields
$\langle \VV' \cup \WW' \rangle_K = \langle \VV \rangle^+_K$. Since the loop
(T4)-(T5) has been executed,\linebreak
$\langle \VV' \cup \WW \rangle_K = \langle \VV' \cup \WW' \rangle_K
\cap \langle \UU \rangle_K$. After exiting
the loop (T3)-(T6) in (T6), $\WW = \emptyset$ and therefore
$\langle \VV' \rangle_K = \langle \VV \rangle^+_K \cap \langle \UU \rangle_K$.
It remains to show $\langle \VV \rangle_K = \langle \VV' \rangle_K$. This
is implied by $\bigcup_{v \in \VV} \{ \MT_{\VV}(v) \}
= \bigcup_{v \in \VV'} \{ \MT_{\VV'}(v) \}$ and $\WW = \emptyset$ during the
last iteration of (T3)-(T6).
%%\footnote{$(i) \VV
%%\subseteq \langle \VV' \cup \WW' \rangle_K$ by construction and
%%$(ii) v \in \VV \Rightarrow
%%v = \sum \alpha_i v_i \in \langle \VV' \rangle_K$ else contradiction.}.\\

The loop (T3)-(T8) is finite: Consider the case that the algorithm
does not terminate in any iteration in (T7). Let $i \in \{ 1, \ldots, n \}$.
$I$ is zero-dimensional thus $I \cap K[x_i] \ne \emptyset$, i.e. there exists
a $p = x_i^d + \ldots + a_1 x_i + a_0 \in I$ where $d \in \NN$ and
$a_0, \ldots, a_{d-1} \in K$. Let $p = h_1 f_1 + \ldots + h_s f_s$ where
$h_j \in P$ where $1 \leq j \leq s$. Since each execution of (T8) strictly
enlarges $\UU$, the relation $h_j f_j \in \langle \UU \rangle_K$ must hold
for all $j \in \{ 1, \ldots, s \}$ after finitely many iterations which yields
$p \in \langle \VV \rangle_K$. Now let $p = x_i^d + \ldots + a_1 x_i + a_0 =
b_1 v_1 + \ldots + b_r v_r$ where $b_1, \ldots, b_r \in K \setminus \{ 0 \}$ and
$v_1, \ldots, v_r \in \VV$. Then $x_i^d = \MT_\VV(v_j)$ for a
$j \in \{ 1, \ldots, r \}$: Assume for a contradiction $x_i^d \ne \MT_\VV(v_j)$
for all $j \in \{ 1, \ldots, r \}$. By construction, $\MT_\VV(v_j) \notin
\Supp(v_k)$ where $j, k \in \{ 1, \ldots, r \}$ and $j \ne k$. This yields
$\MT_\VV(v_j) \in \Supp(p) \setminus \{ x^d \}$. Thus $\deg(v_j) =
\deg(\MT_\VV(v_j)) < d$ where $1 \leq j \leq r$ in contradiction to $\deg(p) =
\deg(b_1 v_1 + \ldots + b_r v_r) = d$. Since no marked term will vanish in a
future iteration of loop (T3)-(T8) we conclude that for all $i \in \{ 1,
\ldots, n \}$ there exist powers $d_i \in \NN$ s.t. $x_1^{d_1}, \ldots,
x_n^{d_n}$ appear as marked terms of polynomials of $\VV$ in some iteration.
The order ideals $\OO$ constructed in (T7) contain no multiples of marked
terms. Thus the growth of the order ideals $\OO$ is bounded since
$\TT^n \setminus \{ t \cdot x_i^{d_i}\ |\ t \in \TT^n, 1 \leq i \leq n \}$
is finite. Therefore after finitely many iterations the relation
$\partial \OO \subseteq \UU$ in (T8) must hold and the loop terminates.

Finally, assume that (T9) is being executed: By proposition
\ref{PrpExistenceOfBB} where $\tilde{I} = \langle \VV \rangle_K$ the
$\OO$-border basis of $I$ exists. By construction,
$\Supp(v) \setminus \{ \MT_\VV(v) \} \subseteq \OO$ for all $v \in \VV$.
Thus the selected $g_1, \ldots, g_\nu$ form a $\OO$-border prebasis of $I$.
Proposition \ref{PrpExAndUniqOfBB} now implies that $\{ g_1, \ldots, g_\nu \}$
is the $\OO$-border basis of $I$.
\end{proof}

\begin{rem}
Let $\FF' = \{ f'_1, \ldots, f'_{s'} \} \subset P \setminus \{ 0 \}$ be a set of
marked polynomials such that $\langle \FF' \rangle_P$ is zero-dimensional.
Using the Marked Interreduction, $\FF'$ can
easily be transformed into a set $\FF = \{ f_1, \ldots, f_{s} \} \subset P
\setminus \{ 0 \}$ that fulfills the precondition $\MT_{\FF}(f_i) \notin
\Supp(f_j)$ where $1 \leq i, j \leq s$ and $i \ne j$ of the Border Basis
Algorithm with Term Marking Strategy.
\end{rem}

\begin{rem}
The following examples shows that this variant of the Improved Border Basis
Algorithm indeed allows the computation of border bases that cannot be computed
by the standard (Improved) Border Basis Algorithm.
\end{rem}

\begin{exa}
Let $K = \QQ$ and $P = \QQ[x,y]$. Let $\FF = \{ f_1, f_2, f_3 \} = \{ x^2 + xy -
\frac{1}{2}y^2 - x - \frac{1}{2}y, y^3 - y, xy^2 - xy \}$. The ideal $I =
\langle f_1, f_2, f_3 \rangle$ is equal to the vanishing ideal of
the five points $(-1, 1), (1,1), (0,0), (1,0), (0,-1) \in \mathbb{A}^2(\QQ)$
of example \ref{ExaStdBBAProblem}. We apply algorithm \ref{PrpBBAWithTMS} to
$\FF$.
\begin{enumerate}
  \item[(T1)] The computing universe $\UU$ is equal to
  $\{ 1, x, y, x^2, xy, y^2, xy^2, y^3 \}$.
  \item[(T2)] We may only choose the marked term of $f_1$ freely because the
  degree restriction forces us to mark $y^3$ in $f_2$ and $xy^2$ in $f_3$. Since
  $\deg(f_1) = 2$, one of the terms $x^2, xy$ or $y^2$ can be marked in $f_1$.
  We choose $xy$.\\
  We apply the Marked Interreduction to $\FF$ and choose the
  numeration $t_1 = y^3,  
  t_2 = xy^2, t_3 = xy, t_4 = x^2, t_5 = y^2, t_6 = x$, and $t_7 = y$. We obtain
  $f'_1 = xy + x^2 - \frac{1}{2}y^2 - x - \frac{1}{2}y$ with marked term
  $xy, f'_2 = y^3 - y$ with marked $y^3$, and $f'_3 = xy^2 + x^2 -
  \frac{1}{2}y^2 - x - \frac{1}{2}y$ with marked term $xy^2$. We let
  $\VV = \{ f'_1, f'_2, f'_3 \}$.
  \item[(T3)] a) We mark $x^2y$ in $x v_1$, $xy^2$ in $y v_1$, $xy^3$ in $x
  v_2$, $y^4$ in $y v_2$, $x^2y^2$ in $x v_3$, and $xy^3$ in $y v_3$.\\
  b) We let $\VV = \{ v_1, v_2, v_3 \} = \{ f_1, f_2, f_3 \}$,
  and $\VV^+ \setminus \VV = \{ v'_1, v'_2, v'_3, v'_4, v'_5,$\linebreak
  $v'_6 \}
  = \{ x v_1, y v_1, x v_2, y v_2, x v_3, y v_3 \}$. We apply the Marked
  Interreduction to $\{ v_1, v_2, v_3, v'_1, v'_2, v'_3, v'_4, v'_5, v'_6 \}$
  and choose the numeration $t_1 = xy^3, t_2 = y^4, t_3 = x^2y^2, t_4 = y^3,
  t_5 = xy^2, t_6 = x^2y, t_7 = xy , t_8 = x^3, t_9 = x^2, t_{10} = y^2, t_{11}
  = x$, and $t_{12} = y$ and obtain $f'_1 = xy + x^2 - \frac{1}{2}y^2 - x -
  \frac{1}{2}y$ with marked term $xy$, $f'_2 = y^3 - y$ with marked term $y^3$,
  $f'_3 = xy^2 + x^2 - \frac{1}{2}y^2 - x - \frac{1}{2}y$ with marked term
  $xy^2$, $f'_4 = x^2y - \frac{1}{2}y^2 - \frac{1}{2}y$ with marked term
  $x^2y$, $f'_5 = x^3 - x$ with marked term $x^3$, $f'_6 = xy^3 + x^2 -
  \frac{1}{2}y^2 - x - \frac{1}{2}y$ with marked term $xy^3$, $f'_7 = y^4 -
  y^2$ with marked term $y^4$, and $f'_8 = x^2y^2 - \frac{1}{2}y^2 -
  \frac{1}{2}y$ with marked term $x^2y^2$.\\
  c) We let $\VV' = \{ f'_1, f'_2, f'_3 \}$ and $\WW' = \{ f'_4, f'_5, f'_6,
  f'_7, f'_8 \}$.
  \item[(T4)] $\WW = \emptyset$ because of $x^2y, x^3, xy^3, y^4, x^2y^2 \notin
  \UU$.
  \item[(T5)] We continue with (T6).
  \item[(T6)] We continue with (T7).
  \item[(T7)] $\OO = \{ 1, x, y, x^2, xy, y^2, xy^2, y^3 \}
  \setminus \{ xy, y^3, xy^2 \} = \{ 1, x, y, x^2, y^2 \}$.
  \item[(T8)] $\partial \OO = \{ x^3, x^2y, xy, xy^2, y^3 \} \nsubseteq \UU$.
  We enlarge $\UU$ to $\UU^+ = \{ 1, x, y, x^2, xy,$\linebreak
  $y^2, x^3, x^2y, xy^2, y^3, x^2y^2, xy^3, y^4 \}$ and repeat (T3)-(T8).
  \item[(T3)] We obtain the same results as in (T3) above.
  \item[(T4)] $\WW = \{ x^2y - \frac{1}{2}y^2 - \frac{1}{2}y, x^3 - x, xy^3
  + x^2 - \frac{1}{2}y^2 - x - \frac{1}{2}y, y^4 - y^2, x^2y^2 - \frac{1}{2}y^2
  - \frac{1}{2}y \} = \WW'$.
  \item[(T5)] It is not necessary to enlarge $\UU$ thus we continue with (T6).
  \item[(T6)] Since $\WW \ne \emptyset$ we replace $\VV$ by $\{ xy +
  x^2 - \frac{1}{2}y^2 - x - \frac{1}{2}y, y^3 - y, xy^2 + x^2 - \frac{1}{2}y^2
  - x - \frac{1}{2}y, x^2y - \frac{1}{2}y^2 - \frac{1}{2}y, x^3 - x, xy^3 + x^2
  - \frac{1}{2}y^2 - x - \frac{1}{2}y, y^4 - y^2, x^2y^2 - \frac{1}{2}y^2
  - \frac{1}{2}y \}$ and repeat (T3)-(T6). This will lead to $\WW = \emptyset$
  in (T6) and we leave the loop (T3)-(T6).
  \item[(T7)] $\OO = \{ 1, x, y, x^2, xy, y^2, x^3, x^2y, xy^2, y^3, x^2y^2,
  xy^3, y^4 \} \setminus \{ xy, y^3, xy^2, x^2y,$\linebreak
  $xy^3, y^4, x^2y^2, x^3 \} = \{ 1, x, y, x^2, y^2 \}$.
  \item[(T8)] $\partial \OO = \{ x^3, x^2y, xy, xy^2, y^3 \} \subseteq \UU$ and
  we leave the loop (T3)-(T8).
  \item[(T9)] We select $g_1 = x^3 - x$, $g_2 = x^2y
  - \frac{1}{2}y^2 - \frac{1}{2}y$, $g_3 = xy + x^2 - \frac{1}{2}y^2 - x
  - \frac{1}{2}y$, $g_4 = xy^2 + x^2 - \frac{1}{2}y^2 - x - \frac{1}{2}y$, and
  $g_5 = y^3 - y$ with marked terms $x^3$, $x^2y$, $xy$, $xy^2$, and $y^3$ from
  $\VV$ and output $g_1, \ldots, g_5$ as well as $\OO = \{ 1, x, y, x^2, y^2
  \}$.
\end{enumerate}
Finally, we note that we obtained the $\OO$-border basis of the
ideal $I$ of example \ref{ExaStdBBAProblem} which cannot be computed by the
(Improved) Border Basis Algorithm.
\end{exa}

\begin{exa}
Let $K = \QQ, P = \QQ[x,y,z]$, and $\FF = \{ f_1, f_2, f_3 \} = \{  x^3 +
x - 1,$\linebreak
$y^2 + yz + z^2 + xz + x^2, z^3 + x^2z + xyz - y \}$ be a set
of marked polynomials where $\MT_\FF(f_1) = x^3, \MT_\FF(f_2) = xz$, and
$\MT_\FF(f_3) = x^2z$. The ideal $I = \langle f_1, f_2, f_3 \rangle$ is
zero-dimensional and $\dim_\QQ(P/I) = 18$. We apply algorithm
\ref{PrpBBAWithTMS} to $\FF$ by using the \DegLex term ordering on
$\TT^3$ to enumerate the terms of the sets $\{ t_1, \ldots, t_l \}$ each time
the Marked Interreduction is executed. Then the computation yields an
$\OO$-border basis of $I$ where $\OO = \{ 1, z, y, x, z^2, yz, y^2, xy, x^2,
z^3, yz^2, y^2z, y^3,$\linebreak
$x^2y, z^4, yz^3, y^2z^2, z^5 \}$ and the marked
terms of the 24 border basis polynomials are $xz, x^2z, x^3, xyz, xy^2, xz^2,
x^2yz, x^3y, xy^3, x^2y^2, xy^2z, xyz^2, y^4, y^3z, xz^3, xyz^3,$\linebreak
$xy^2z^2, y^3z^2, y^2z^3, yz^4, xz^4, xz^5, yz^5,$ and $z^6$. Since $\OO$
contains $x^2, y^2$ and $z^2$ and $f_2 = y^2 + yz + z^2 + xz + x^2 \in I$ the
intersection $\LTs(I) \cap \OO$ can never be empty for any term ordering
$\sigma$ on $\TT^3$. Thus the computed border basis cannot arise from a term
ordering on $\TT^3$.
\end{exa}

\begin{rem}
An implementation of the Border Basis Algorithm with Term Marking Strategy is
available under the (function) name \texttt{BB.BBasisForMP} in the computer
algebra system \apcocoa (cf. \cite{ApCoCoA, CoCoA}).
\end{rem}

\begin{rem}\label{RemDCTOAlwaysWorks}
Let $\sigma$ be a degree compatible term ordering on $\TT^n$. As noted before
in example \ref{ExaDegCompTOIsMarking} $\LTs$ is a marking for any $\FF \subseteq
P \setminus \{ 0 \}$. The input set $\FF$ of the Border Basis Algorithm with
Term Marking Strategy can thus be considered being marked by $\LTs$. Now assume
that each time the Marked Interreduction is applied to a set $\{ f_1, \ldots,
f_s \} \subset P \setminus \{ 0 \}$ during the execution of the Border Basis
Algorithm with Term Marking Strategy the set
$\mathop{\bigcup}_{i=1}^s \Supp(f_i) = \{ t_1, \ldots, t_l \}$ is enumerated in
such a way that $t_1 >_\sigma \ldots >_\sigma t_l$ holds. Then the resulting
output set of polynomials $\FF'$ will be marked accordingly to $\LTs$, i.e.
each marked term of a polynomial of $\FF'$ will be the leading term of this
polynomial with respect to $\sigma$. This means that in this case the output of
the Border Basis Algorithm with Term Marking Strategy will be the same as the
output of the Improved Border Basis Algorithm \cite[Proposition
21]{KKCompBBases}, namely the $\OOsI$-border basis of $\langle \FF \rangle_P$.
\end{rem}

\begin{rem}\label{RemAlgoCanStopInT7}
The algorithm can indeed encounter the situation that in step (T7) the set
$\OO$ is not an order ideal. In the following we briefly discuss to reasons why
the algorithm cannot produce a border basis for every given input. For this, we
assume that the algorithm reaches step (T7) and let $\tilde{\OO} = \UU
\setminus \{ t \cdot \MT_\VV(v)\ |\ v \in \VV, t \in \TT^n \}$ and $\OO = \UU
\setminus \{ \MT_\VV(v)\ |\ v \in \VV \}$. Note that $\tilde{\OO}$ always forms
an order ideal whereas $\OO$ is an order ideal if and only if $\OO =
\tilde{\OO}$.
\begin{enumerate}
  \item[a)] One reason for the algorithm to terminate in step (T7) is that the
  order ideal $\tilde{\OO}$ is too small to satisfy $|\tilde{\OO}| =
  \dim_K(P/I)$. In this case $\OO \setminus \tilde{\OO}$ will contain at least
  one term that does not appear as a marked term of any polynomial in $\VV$ and
  $\OO$ will not form an order ideal. An example for this
  case is given in \ref{ExaTooSmallOI} below.
  \item[b)] Another reason why the algorithm terminates in step (T7) is the
  scenario when $\tilde{O}$ outlines an order ideal for which no border basis
  of $I$ exists.
\end{enumerate}
\end{rem}

\begin{exa}\label{ExaTooSmallOI}
Let $P = \QQ[x, y]$ and $\FF = \{ f_1, f_2, f_3 \} = \{ x^3, y^3, x^2 + xy
+ y^2 \}$. The ideal $I = \langle \FF \rangle$ is zero-dimensional and
$\dim_{\QQ}(P/I) = 6$. We apply the Border Basis Algorithm with Term Marking
Strategy to $\FF$ and choose to mark the terms $x^3$ in $f_1$, $y^3$ in $f_2$,
and $xy$ in $f_3$ in step (T2). Then, already at the beginning of step (T3),
the set $\tilde{\OO} = \UU \setminus \{ t \cdot \MT_\VV(v)\ |\ v \in \VV, t \in
\TT^n \} = \{ 1, x, y, x^2, y^2 \}$ constitutes an order ideal which is too
small to support a border basis of $I$. As none of the marked terms of $\VV$
will vanish during the following computations the set $\tilde{\OO}$ may only
shrink further (which it does not do in this example) but cannot become larger.
Eventually the algorithm will produce the set $\OO = \{ 1, x, y, x^2, y^2, xy^2
\}$ in step (T7) which is apparently not an order ideal. Here $\OO \setminus
\tilde{\OO} = \{ xy^2 \}$ reveals which term of $\OO$ is a multiple of a marked
term of a polynomial in $\VV$.
\end{exa}

\begin{rem}\label{RemBBAWithTMSBacktracking}
The algorithm can be equipped with a backtracking strategy to exhaustively check
if a given marking of the input polynomials allows the computation of a border
basis. For this to work, the additional book keeping functionality described in
remark \ref{RemMIBacktracking} must be included in the Marked Interreduction. If
it is then detected in step (T7) that the set $\OO$ is not an order ideal, the
algorithm can successively go backwards through the log entries of the Marked
Interreduction computations, choose different pivot indices $p_k$ in step (6),
and redo the computation from those points on. If the log entries of the Marked
Interreduction are exhausted it is then clear that the given marking of the
input polynomials does not allow the computation of a border basis.
\end{rem}

\begin{rem}
The outcome of the application of the Border Basis Algorithm with Term Marking
Strategy highly depends on the input polynomials and the given marking of them.
It is clear that if the support of the given input polynomials does not allow to
create a marking of the polynomials where the marked terms are terms that can
never be a leading term with respect to any term ordering then it is less likely
that the computation will yield a border basis that does not arise from a term
ordering. On the other hand, in this situation it is also less likely that the
algorithm stops in step (T7), especially if the backtracking described in remark
\ref{RemBBAWithTMSBacktracking} is used. From an experimental point of view it
is thus more fruitful to apply the algorithm to input polynomials that can be
marked in many different ways.
\end{rem}

\bigskip

%%%%%%%%%%%%%%%%%%%%%%%%%%%%%%%%%%%%%%%%%%%%%
% Relation to other Border Basis Algorithms %
%%%%%%%%%%%%%%%%%%%%%%%%%%%%%%%%%%%%%%%%%%%%%
\section{Relation to other Border Basis Algorithms}\label{SecRelToOtherAlgos}
Let $I \subset P$ be a zero-dimensional ideal. In
\cite{MTGenNormFormAndPolySysSolv}, Mourrain and Tr\'{e}buchet introduced
a very general algorithm which allows the computation of a quotient basis
$\mathcal{B}$ of the $K$-vector space $P/I$. In addition, a set of reducing
rules which allow projection onto $\langle \mathcal{B} \rangle_K$ along $I$ is
produced. In this algorithm a choice function refining a reducing graduation
$\gamma$ (cf. \cite[Definition 2.7]{MTGenNormFormAndPolySysSolv}) is used to
determine the resulting basis $\mathcal{B}$. For certain choices of $\gamma$ the
output $\mathcal{B}$ of this algorithm is an order ideal and the set of
reducing rules is a $\mathcal{B}$-border basis of $I$. We observe the following
fundamental difference between the algorithm of Mourrain and Tr\'{e}buchet and
algorithm \ref{PrpBBAWithTMS}: The set $\mathcal{B}$ in the algorithm of
Mourrain and Tr\'{e}buchet may grow and shrink as needed during the
computation. In contrast to this the set $\OO$ in algorithm \ref{PrpBBAWithTMS}
may only shrink in each iteration of step $(T7)$. The latter behaviour ensures
that no marked term of any marked polynomial in the initial input set $\FF$ can
appear in the order ideal $\OO$.

\smallskip

Let $I$ be generated by a finite set of polynomials $\FF \subset P$. Braun and
Pokutta presented another more general algorithm to compute an $\OO$-border
basis of $I$ for an order ideal $\OO$ in \cite{BPBBasesAndOIPolyhedralChar}. In
contrast to algorithm \ref{PrpBBAWithTMS} not only one specific $\OO$-border
basis is computed but instead an $L$-stable span of $\langle \FF \rangle_K$ (cf.
\cite[Definition 10]{KKCompBBases}) where $L = \langle \TT^n_{\leq d} \rangle_K
= \langle \{ t \in \TT^n\ |\ \deg(t) \leq d \} \rangle_K$ for some $d \in \NN$
is produced which contains all possible $\OO$-border bases of $I$. After
choosing an admissible order ideal $\OO$ the polynomials of the $\OO$-border
basis are then selected from this $K$-vector space in the last step of the
algorithm by a basis transformation. Due to the dependency on the shape and the
marking of the input polynomials, the application of the Border Basis Algorithm
with Term Marking Strategy always results in a very specific computation which
means that it does not necessarily produce such a $K$-vector space during its
execution. This behavior can result in a faster running time at the expense of
generality as shown in the following example.

\begin{exa}
Let $K = \QQ, P = \QQ[x,y], \FF = \{ x^2 - y, x^2y + y^3 - x - y \}$, and $I =
\langle \FF \rangle$. Let $\sigma$ denote the standard \Lex term
ordering on $\TT^2$ where $x >_\sigma y$ and let $\tau$ denote the
\Lex term ordering on $\TT^2$ where $x <_\tau y$. Since the reduced
$\sigma$- and $\tau$-Gr\"{o}bner bases of $I$ are $\{ x - y^3 - y^2 + y, y^6 +
2y^5 - y^4 - 2y^3 + y^2 - y \}$ and $\{ y - x^2, x^6 + x^4 - x^2 - x \}$,
respectively, algorithm $4.3$ of \cite{BPBBasesAndOIPolyhedralChar} must produce
the computing universe $L = \TT^2_{\leq 6}$ for a stable span computation during
its execution. Because of the shape of the polynomials of $\FF$ the Border Basis
Algorithm with Term Marking Strategy cannot be used to compute neither the
$\OOsI$-border basis nor the $\mathcal{O}_\tau\{I\}$-border basis of $I$. It can
be shown that regardless of the chosen marking of $\FF$ and the
chosen enumeration of terms for the application of the Marked Interreduction
the result of the computation in this case will always be the same $\OO$-border
basis of $I$. But in contrast to the computing universe $L = \TT^2_{\leq 6}$ as
in the case of algorithm $4.3$ of \cite{BPBBasesAndOIPolyhedralChar} the final
computing universe during the execution of the Border Basis Algorithm with Term
Marking Strategy is the order ideal $\UU \subset \TT^2_{\leq 4}$ spanned by
$y^4, xy^3, x^2y^2,$ and $x^3y$.
\end{exa}

\bigbreak

%%%%%%%%%%%%%%%%
% Bibliography %
%%%%%%%%%%%%%%%%

\end{document}